\documentclass[12pt]{amsart}
\usepackage{amsmath,amssymb,amscd,amsthm,amscd,
mathrsfs,enumerate}
\usepackage[latin1]{inputenc}

\usepackage[english]{babel}
\usepackage{color}%
\newcommand{\spec}{{\rm Spec}}

\DeclareMathOperator{\tspec}{{\it t}-Spec}
\DeclareMathOperator{\tmax}{{\it t}-Max}

\newcommand{\Int}{{\rm Int}}

\newcommand{\ms}{\mathscr}

\newcommand{\ad}{{\rm Cl}}
\newcommand{\f}{\mathfrak}

\newcommand{\Max}{{\rm Max}}

\newcommand{\mb}{\mathbb}

\newcommand{\z}{{\ldots}}

\newcommand{\U}{\ms U}

\newtheoremstyle{break1}
  {9pt}
  {9pt}
  {}
  {}
  {\textbf}
  {.}
  {.7em}
  {}
\theoremstyle{break}
\newtheorem{thm}{ \textbf{Theorem}}[section]

\newtheorem{cor}[thm]{ \textbf{Corollary}}
\newtheorem{lem}[thm]{ \textbf{Lemma}}
\newtheorem{prop}[thm]{ \textbf{Proposition}}

\theoremstyle{definition}
\newtheorem{ex}[thm]{ \textbf{Example}}
\newtheorem{oss}[thm]{Remark}
\theoremstyle{remark}

\def\SQ{\mathbb Q}    
\def\SZ{\mathbb Z}    

\begin{document}

\title[]
{On a topological characterization of Pr\"ufer $v$-multiplication domains among essential domains}
\author{Carmelo Antonio Finocchiaro}
\address[Carmelo Antonio Finocchiaro]{Dipartimento di Matematica e Fisica\\ Universit\`{a}
degli studi Roma Tre\\ Largo San Leonardo Murialdo 1, 00146 Roma,
Italy}
\email{carmelo@mat.uniroma3.it}
\author{Francesca Tartarone}
\address[Francesca Tartarone]{Dipartimento di Matematica e Fisica\\ Universit\`{a}
degli studi Roma Tre\\ Largo San Leonardo Murialdo 1, 00146 Roma,
Italy} \email{tfrance@mat.uniroma3.it}
 

\begin{abstract}
 In this paper we characterize the Pr\"ufer $v$-multiplication domain   as a class of essential domains verifying an additional property on the closure of some families of prime ideals, with respect to the constructible topology.

\end{abstract} \maketitle
\section*{Introduction}

 The notion of   Pr\"ufer domain, introduced by H. Pr\"ufer in 1932, plays a central role in the theory of integrally closed   domains. In fact  it globalizes  the concept of   valuation domain in the sense that a domain is Pr\"ufer if and only if it is locally a valuation domain  (i.e. all its localizations at  prime  ideals are valuation domains). There is a wide literature about the investigation of the multiplicative structure of  ideals in Pr\"ufer domains  (for a deeper insight on recent developements on this topic, see   \cite{fo-ho-lu}, \cite{olb}).   The notion of  \emph{Pr\"ufer $v$-multiplication domain} (briefly, P$v$MD) was introduced to enlarge the class of    Pr\"ufer domains (for instance,    two-dimensional regular   domains  are P$v$MD but not Pr\"ufer). More precisely, an integral domain is a P$v$MD if and only if it is $t$-locally a valuation domain,  i.e.  each localization at $t$-prime   ideals is a valuation domain  (Section 1). Here we just point out that the condition of being $t$-locally a valuation domain is certainly weaker than being locally a valuation domain because it involves a subset of the prime spectrum of a domain.   
Other interesting examples of P$v$MD's, besides Pr\"ufer domains, are for instance $\SZ[X]$ and, more generally, Krull domains. 


M. Griffin in \cite{Gr} gives a very simple characterization of the P$v$MDs with the $t$-finite character (i.e., each nonzero element of $D$ is contained in finitely many $t$-maximal ideals). In this case they are exactly the essential domains with the $t$-finite character (Theorem \ref{griffin}).

But the essential property for a domain $D$ is not, in general, equivalent to saying that $D$ is a P$v$MD. An important example of this fact is given by W. Heinzer - J. Ohm in \cite{ho}. We have gone through this construction in order to understand what is missing in this essential domain that makes it not to be a P$v$MD. Then we used this observation to give a general characterization of P$v$MDs among the essential domains.

The central result of this paper is Theorem \ref{pvmdchar} in which we  describe exactly P$v$MDs as a subclass of essential domains that verifies an additional condition  regarding ultrafilter limits    of   suitable families of prime ideals. 

This theorem is on the one hand a generalization to any essential domain of the above-mentioned result by Griffin   on domains with $t$-finite character and, on the other hand, it gives a topological explanation of what goes wrong with Heinzer-Ohm example  of an essential domain that is not a P$v$MD. 

In  Corollary \ref{sopranelli} we   compare the P$v$MD property among domains with different quotient fields. In particular we give a result in the case in which these quotient fields $K$ and $L$ form an algebraic extension $K \subseteq L$.

An interesting  still open question is when a family of  P$v$MDs $\{D_i:i\in I\}$ is such that the intersection $D = \bigcap_{i \in I}D_i$ is  a P$v$MD. This  does not happen even in very easy cases like the intersection of two P$v$MD's (for instance domains of the type $V \cap \SQ[X]$, where $V$ is a valuation overring of $\SZ[X]$, are quite often non-P$v$MD).
In Theorem \ref{intersezioni} and Corollary \ref{finite intersection}   we partially answer this question. In particular, we show that if the family is finite and   $D $ is  \lq\lq essential with respect to each $D_i$"  then $D$ is   a P$v$MD.

\medskip
An interesting application of Theorem \ref{pvmdchar} is given in Section \ref{applications} with regard to the ring of integer valued polynomials  over a domain $D$,   $\Int(D) = \{f \in K[X]: f(D) \subseteq D\}$. 

Both the problems of characterizing when $\Int(D)$   is Pr\"ufer or P$v$MD have been investigated in the last twenty years (see, for instance, \cite{chabert, loper, Calota}).

Here we discuss   the P$v$MD property of $\Int(D)$. In Theorem \ref{pvmdint}  we refine the general characterization  of a domain   $D$ such that   $\Int(D)$  is a P$v$MD given in \cite{Calota}. More precisely, we show that  one of the three equivalent conditions of \cite[Theorem 3.4]{Calota} posed on $D$ can be deleted by putting an extra-hypothesis on the localizations $\Int(D_P)$ (for $P \in \tspec(D)$).

Most of the results presented in this paper are topological in nature and their proofs are often based on techniques involving the constructible topology. For relevant contributions on this circle of ideas see, for instance, \cite{fifolo}, \cite{olb3}.

\section{Preliminaries}\label{basic}
With the term \emph{ring} we will mean always a commutative ring with identity and, as usual, we denote by $\spec(A)$ the set of all prime ideals of a ring $A$. For any ring homomorphism $f:A\longrightarrow B$, we shall denote by $f^\star:\spec(B)\longrightarrow \spec(A)$ the canonical  map, { induced by $f$}.

\vspace{0.5cm}
\textsc{\underline{Ultrafilter limit points}} 
\vspace{0.5cm}

Given a set $X$, we recall that an \emph{ultrafilter on $X$} is a  collection $\ms U$ of subsets of $X$ such that: 
\begin{enumerate}[\hspace{0.6cm}\rm (1)]
\item $\emptyset\notin \ms U$.
\item If $Y,Z\in \ms U$, then $Y\cap Z\in \ms U$.
\item If $Y\in \ms U$ and $Y\subseteq Z\subseteq X$, then $Z\in \ms U$.
\item For any $Y\subseteq X$, either $Y\in \ms U$ or $X-Y\in \ms U$.
\end{enumerate}
We remind in the following remark some basic properties of ultrafilters that will be useful. 
\begin{oss}\label{beginning-ultra} Let $X$ be a set.
\begin{enumerate}[\hspace{0.6cm}\rm (1)]
\item If $\ms F$ is a collection of sets with the finite intersection property, then $\ms F$ extends to some ultrafilter $\ms U$ on $X$ (i.e., $\ms U\supseteq \ms F$).
\item  If $x\in X$, the collection of sets $\ms U_x:=\{Y\subseteq X:x\in Y\}$ is an ultrafilter on $x$, called \emph{a principal ultrafilter.} From the definitions, it easily follows   that an ultrafilter is trivial if and only if it contains a finite set. 
\item An ultrafilter $\ms U$ on $X$ is \emph{nontrivial} if $\ms U\neq \ms U_x$, for any $x\in X$. A straightforward application of Zorn's Lemma shows that  $X$ admits non principal ultrafilters if and only if it is infinite.
\end{enumerate}

\end{oss}

Now, let $A$ be a ring. Unless otherwise specified, we endow $\spec(A)$ with the Zariski topology, whose closed sets are of the form $$V(\f a):=\{\f p\in \spec(A): \f a\subseteq \f p\},$$ for any ideal $\f a$ of $A$.  For any $Y\subseteq \spec(A)$, we shall
denote by $\ad^{\rm c}(Y)$ the closure of $Y$, with respect to the
constructible topology, that is the smallest  topology for which any set of the form 
$$
D(f):=\{\f p\in \spec(A):f\notin \f p\} \qquad (f\in A)
$$
is clopen. It follows easily by definitions that a basis of clopen sets for the constructible topology is 
$$
\{D(f)\cap V(\f a):f\in A, \f a \mbox{ finitely generated ideal of }A\}
$$ Recently, a relation between the constructible topology and the notion of \emph{ultrafilter limit point} has been shown independently in \cite{fl} and \cite{fi}. More precisely, let $Y$ be a nonempty subset of $\spec(A)$ and let $\ms U$ be an ultrafilter on $Y$. By \cite[Lemma 2.4]{Calota}, the set
$$
Y_{\ms U}:=\{a\in A:V(a)\cap Y\in \ms U\}
$$
is a prime ideal of $A$, called \emph{the ultrafilter limit point of $Y$, with respect to $\ms U$}.  By \cite[Theorem 8]{fl} and \cite[Corollary 2.16]{fi}, a set is closed  with respect to the constructible topology if and only if it contains all of its ultrafilter limit points. Moreover, by \cite[Proposition 2.12]{fi}, we have 
 \label{diamond} $$(\diamond) \qquad
\ad^{\rm c}(Y)=\{Y_{\ms U}:\ms U \mbox{ ultrafilter on } Y\}
$$
for every subset $Y$ of $\spec(A)$.

\vspace{0.5cm}
\textsc{\underline{The $t$-operation}} 
\vspace{0.5cm}

  Given an integral domain $D$ with quotient field $K$ we have the following usual terminology and definitions: for each nonzero (fractional) ideal $I$ of $D$ the \textit{divisorial closure} of $I$ is the ideal $I^v = (D:(D:I))$, where $(D \colon I) := \{x \in K : xI \subseteq D\}$.

The \textit{$t$-closure} of $I$ is
$$I^t=\bigcup\{J^v:J \textrm{ is finitely generated ideal and }  J \subseteq I\}$$

The ideal $I$ is called a \textit{$t$-ideal} if   either $I=(0)$ or  $I = I^t$ and it is a \textit{$t$-prime} if it is prime and a $t$-ideal (usually the notion of a $t$-ideal is given for  \emph{nonzero} fractional ideals, but here it will be convenient to declare $(0)$ a $t$-ideal, by definition).
A $t$-maximal ideal is a $t$-ideal which is maximal among the proper $t$-ideals of $D$. A $t$-maximal ideal is   $t$-prime and
a proper $t$-ideal is always contained in a $t$-maximal ideal. We denote by $\tmax(D)$ the set of the $t$-maximal ideals of $D$ and by $\tspec(D)$ the set of $t$-prime ideals of $D$. For background material on $t$-operation  see, for instance,  \cite{gilmer, jaffard} 

\section{Main Results}
Let $D$ be an integral domain. 
A valuation overring of $D$ is said
to be \emph{essential}  for $D$ if it is a localization of $D$. A prime ideal of $D$ is \emph{essential} if it is the center of an essential valuation overring of $D$.  A collection
$\mathcal V$ of overrings of $D$ is said to be \emph{an essential
representation of $D$} if $D =\bigcap\{
V:V\in\mathcal V\}$ and each member of  $\mathcal V$ is essential for $D$. Recall that $D$ is said to be \emph{essential} if it has an essential representation.  Denote by $\mathcal E(D)$ \emph{the essential prime spectrum of $D$}, i.e., 
$$
\mathcal E(D):=\{\f p\in \spec(D): D_{\f p}\mbox{ is a valuation domain }\}
$$

\begin{oss}\label{basic-facts-toperation} We recall the following well-known facts:
\begin{enumerate} 

\item Any domain $D$ can be represented as $D = \bigcap_{M \in \tmax(D)}D_M$ \cite[Proposition 4]{Gr}. 
\item An integral domain is a Pr\"ufer domain if and only if every ideal is a $t$-ideal \cite[Proposition 34.12]{gilmer}.   In particular, every ideal of a valuation domain is a $t$-ideal.
\item A P$v$MD is always essential because $D_P$ is a valuation domain for each $M \in \tmax(D)$  (\cite[Theorem 3.2]{ka}).
\item For any integral domain $D$ the following inclusion $\mathcal E(D)\subseteq \tspec(D)$ holds, by \cite[Lemma 3.17]{ka}. 
\item There exist essential domains that are not P$v$MD. An example is given by W. Heinzer - J. Ohm (\cite{ho}, see also the following Example \ref{example-heinzer}).

\item A domain $D$ has the $t$-finite (resp. finite) character if each nonzero element $x \in D$ belongs to finitely many $t$-maximal (resp. maximal) ideals.

\noindent P$v$MD's may not have the $t$-finite character: for instance take $\SZ + X\SQ[X]$.

\item By  \cite[Lemma 2.4 \& Proposition 2.5]{Calota},      every  nonzero ultrafilter limit  of a family of $t$-prime ideals is a $t$-prime ideal. Since we have set $(0)$ to be a $t$-ideal, we have that every    ultrafilter limit  of a family of $t$-prime ideals is a $t$-prime ideal,  that is, $\tspec(D)$ is closed, with respect to the constructible topology \cite[Theorem 8]{fl}.  
 
{ Thus, if $D$ is a P$v$MD, then $\mathcal E(D)$ is closed  with respect to the constructible topology, since $\mathcal E(D)=\tspec(D)$}.

\end{enumerate}
\end{oss}

In Theorem \ref{pvmdchar} we characterize P$v$MDs in terms of the closure (with respect to the constructible topology) of a suitable subset of $\mathcal E(D)$.

We say that a
collection of overrings $\mathcal O$ of $D$ is \emph{locally finite}
if  for any nonzero element $x \in D$  the set $\{B\in\mathcal O:
x\mbox{ is not invertible in }B\}$ is finite. Recall that an
integral domain is \emph{a Krull-type domain} if it is an essential
domain  and it has a locally finite essential representation. The
following result characterizes Krull-type domain. 

\begin{thm}{\rm (\cite[Proposition 4, Theorems 5 and 7]{Gr})}\label{griffin} Let $D$ be an integral domain.  
  The following conditions are equivalent.
\begin{enumerate}[\hspace{0.6cm}\rm (i)]
\item $D$ is a Krull-type domain.
\item $D$ is a P$v$MD with $t$-finite character.
\end{enumerate}
 
\end{thm}

The following example is given in \cite{ho} and it is a construction of an essential domain that is not a P$v$MD.
It is not easy to find a domain with these properties, and our aim is to go through this construction  by
giving evidence to some topological aspects   that  will be central in the next Theorem \ref{pvmdchar}.

\begin{ex}{\rm (\cite{ho})}\label{example-heinzer}
Let $K$ be a field and let $X_0,X_1\z,X_n,\z,T,U$ be an infinite and countable collection of intedeterminates over $K$. Set $\mathcal X:=\{X_n:n\in \mb N\}$ and consider the Krull domain $R:=K(\mathcal X)[T,U]_{(T,U)}$. Moreover, for any $i\in\mb N$, let $v_i$ be the valuation on $L:=K(\mathcal X,T,U)$ such that $v_i(K(\{X_j:j\neq i\}))=\{0\}$ and $v_i(X_i)=v_i(T)=v_i(U):=1$ (define $v_i$ on polynomials in the canonical way, i.e., just by taking the infimimum of the value of each monomial, and extend it to $L$). For any $i\in \mb N$, let  $V_i$ be the DVR associated to $v_i$, let $\f m_i$ be its maximal ideal,  and set $D:=R\cap\bigcap_{i\in \mb N}V_i$. In \cite{ho}, the authors show that $D$ is an essential domain that is not a P$v$MD. More precisely, they shows that 
$$
Y:=\{\f p\cap D: \f p\mbox{ height-one prime of }R\}\cup \{\f m_i\cap D:i\in \mb N\}
$$
is a collection of essential prime ideals of $D$. 

 We will give now a new proof of the fact that $D$ is not a P$v$MD and it will help to understand the characterization given in   the following  Theorem \ref{pvmdchar}. As a matter of fact, set 
 $$\ms F:=\{V(f)\cap Y:f \in D\cap (T,U)K(\mathcal X)[T,U]_{(T,U)}\}$$ 
 and take finitely many elements $f_1,\z,f_h\in D\cap (T,U)K(\mathcal X)[T,U]_{(T,U)}$. Then there is a natural integer $n$ such that $$(f_1,\z,f_n)D\subseteq D\cap (T,U)K(X_0,\z,X_n)[T,U]_{(T,U)}.$$
 Furthermore, it is straightforward to show that the inclusion $D\cap (T,U)K(X_0,\z,X_n)[T,U]_{(T,U)}\subseteq \f m_i\cap D$ holds, for any $i>n$. It follows that $\f m_i\cap D\in V((f_1,\z,f_n)D)\cap Y$, for any $i>n$, i.e., the collection of sets $\ms F$ has the finite intersection property. Then there is an ultrafilter $\ms U$ on $Y$ such that $\ms F\subseteq \ms U$. By definition, the ultrafilter limit point $Y_{\ms U}:=\{f\in D:V(f)\cap Y\in \ms U\}$ satisfies the inclusion $D\cap (T,U)K(\mathcal X)[T,U]_{(T,U)}\subseteq Y_{\ms U}$. It follows $D_{Y_{\ms U}}\subseteq R$, and then $D_{Y_{\ms U}}$ is not a valuation domain. Moreover, keeping in mind the equality $(\diamond)$ and Remark \ref{basic-facts-toperation}, we have 
 $$Y_{\ms U}\in \ad^{\rm c}(Y) \subseteq \ad^{\rm c}(\tspec(D))=\tspec(D).$$
 Thus $D$ is not a P$v$MD. 
 
 Observe that we have found the bad $t$-prime ideal $Y_{\ms U}$ that makes $D$ fail to be a P$v$MD in the closure of the set of the centers of an essential representation of $D$. 
\end{ex}

In view of Theorem \ref{griffin} and the previous example, the following question arises naturally: let $D$ be an essential domain with an essential representation $\mathcal V$. Is it possible to
put on $\mathcal V$ an extra condition,  weaker than locally finiteness, in order to get that $D$ is a P$v$MD? 

The following Theorem \ref{pvmdchar} will give a positive answer to this question. 
\begin{thm}\label{pvmdchar}
Let $D$ be an integral domain and $\mathcal E(D)$ be the essential prime spectrum of $D$. Then, the following conditions are equivalent.
\begin{enumerate}[\hspace{0.6cm}\rm (i)]
\item $D$ is a PvMD.
\item $D$ is an essential domain and there is an essential representation $\mathcal V:=\{D_{\f p}:\f p\in Y\}$ of $D$,
for some $Y\subseteq \spec(D)$, such that $\ad^{\rm c}(Y)\subseteq
\mathcal E(D)$.
\end{enumerate}
\end{thm}
\begin{proof} (i)$\Longrightarrow$(ii). Assume that $D$ is a P$v$MD and take $Y:=\tspec(D)$. Applying \cite[Corollary 2.10]{fl2}, it follows easily that $Y$ is closed, with respect to the constructible topology and, by assumption $Y\subseteq \mathcal E(D)$. Finally, it sufficies to note that $\{D_{\f p}:\f p\in Y\}$ is an essential representation of $D$.

(ii)$\Longrightarrow$(i). Let $\f m$ be a $t$-maximal ideal of $D$ and set
$$
\ms F:=\{V(x)\cap Y:x\in\f m\}
$$
We claim that $\ms F$ has the finite intersection property. If not,
there exist elements $x_1,\z,x_n\in \f m$ such that
$V(x_1,\z,x_n)\cap Y=\emptyset$. Thus, if $\f a:=(x_1,\z,x_n)D$, for any $\f p\in Y$ there is an element $d\in \f a-\f p$. For any $x\in (D:\f a)$ we have $dx\in D$, that is, $x\in D_{\f p}$. Since $\mathcal V$ is an essential representation, we infer that $(D:\f a)=D$, i.e. $\f a^v=D$. Since $\f a\subseteq \f m$ and $\f m\in \tmax(D)$, we have $\f a^v=\f m^t=\f m=D$, a contradition.
 Since $\ms F$ has
the finite intersection property, we can pick an ultrafilter $\ms U$
on $Y$ extending $\ms F$. Then, it follows by definition $\f
m\subseteq Y_{\ms U}:=\{x\in D:V(x)\cap Y\in \ms U\}$. Now not that,
by \cite[Lemma 3.17(1)]{ka} and \cite[Corollary 2.10]{fl2},
$\ad^{\rm c}(Y)\subseteq \tspec(D)$, and thus $Y_{\ms U}\in
\tspec(D)$. It follows $\f m=Y_{\ms U}$. On the other hand,
$\ad^{\rm c}(Y)\subseteq \mathcal E(D)$, and thus $D_{\f m}=D_{Y_{\ms U}}$
is a valuation domain. The proof is now complete.
\end{proof}
 
\begin{cor}\label{pvmdcharcor} Let $D$ be an essential domain that admits an essential representation $\mathcal V$ such that the set of the centers in $D$ of the valuation domains in $\mathcal V$ is closed, with respect to the constructible topology. Then $D$ is a P$v$MD.
\end{cor}
\begin{proof}
Apply Theorem \ref{pvmdchar}.
\end{proof}
\begin{cor}
An integral domain $D$ is a P$v$MD if and only if $D$ is essential and $\mathcal E(D)$ is closed, with respect to the constructible topology. 
\end{cor}
\begin{proof}
If $D$ is a P$v$MD, then $\mathcal E(D)=\tspec(D)$ and it is closed with respect to the constructible topology. Conversely, if $D$ is essential, then $\{D_{\f p}:\f p\in \mathcal E(D)\}$ is clearly an essential representation of $D$. Since, by assumption, $\mathcal E(D)$ is closed, the conclusion follows  by Corollary \ref{pvmdcharcor}.
\end{proof}
\begin{cor}\label{compactness-zariski}
Let $D$ be an integral domain. Then, the following conditions are equivalent.
\begin{enumerate}[\rm (i)]
\item $D$ is a P$v$MD.
\item $D$ is an essential domain and it admits an essential representation $\{D_{\f p}:\f p\in Y \}$, where $Y\subseteq \spec(D)$ is compact, with respect to the Zariski topology. 
\end{enumerate}
\end{cor}
\begin{proof}
(i)$\Rightarrow$(ii) follows by taking $Y:=\tspec(D)$. 

(ii)$\Rightarrow$(i). Let $Y$ be a compact subspace of $\spec(D)$ such that $\{D_{\f p}:\f p\in Y \}$ is an essential representation of $D$ and set $$Y^{\rm gen}:=\{\f q\in\spec(D):\f q\subseteq \f p,\mbox{ for some }\f p\in Y\}$$
Of course, $\mathcal  V:=\{D_{\f q}:\f q\in Y^{\rm gen}  \}$ is still a representation of $D$. We claim that $\mathcal V$ is also essential since, if $\f q\in Y^{\rm gen}$ and $\f p\in Y$ is such that $\f q\subseteq \f p$, then $D_{\f p}$ is a valuation domain and $D_{\f q}\supseteq D_{\f p}$. The conclusion follows by applying Theorem \ref{pvmdchar} and \cite[Proposition 2.6]{fifolo}. 
\end{proof}
  We give now a natural application of Corollary \ref{pvmdcharcor}.
  
\begin{ex}\label{fiber-product}(see \cite[Theorem 4.1]{fo-ga})
Let $V$ be a valuation domain with residue field   $k$ and let $\pi:V\longrightarrow k$ be the canonical projection. Let  $D$ be a P$v$MD  whose quotient field is $k$. Consider the following pullback diagram:

$$
\CD
R @>>> D \\
@VVV   @VVV \\
V  @>{\pi}>> k
\endCD
$$

\bigskip

We claim that the ring $R:=\pi^{-1}(D)$ is a P$v$MD. 

By \cite[Corollary 1.9]{fo-ga}, $\pi^{-1}(\f p)$ is a $t$-prime ideal of $R$, for any $t$-prime ideal $\f p$ of $D$ and it is easy to check that   $\pi^{-1}(D_{\f p})=R_{\pi^{-1}(\f p)}$.

 Thus, keeping in mind that $D$ is a P$v$MD whose quotient field is $k$, \cite[Theorem 2.4(1)]{fo} implies that the collection $\mathcal V:=\{R_{\pi^{-1}(\f p)}:\f p\in \tspec(D)\}$ is an essential representation of $R$. 

The centers in $R$ of the valuation domains in $\mathcal V$ are the inverse images $\pi^{-1}(\f p)$, for $\f p \in \tspec(D)$.  This set is closed with respect to the constructible topology, by \cite[Chapter 3, Exercise 29]{AM} and Remark \ref{basic-facts-toperation}(7). The conclusion follows by Corollary \ref{pvmdcharcor}.
\end{ex}
 
The following Lemma will be useful to explain why Griffin's Theorem \ref{griffin} follows from Theorem \ref{pvmdchar}.
 
\begin{lem}\label{easy} Let $A$ be a ring and $Y$ be an infinite subset of $\spec(A)$ such
that every nonzero element of $A$ belongs to only finitely many prime ideals in $Y$.
Then, $A$ is an integral domain and $\ad^{\rm c}(Y)=Y\cup \{(0)\}$.
\end{lem}
\begin{proof}
Since $Y$ is infinite, take a non principal ultrafilter  $\ms U$ on $Y$, and let
$$
Y_{\ms U}:=\{x\in A:V(x)\cap Y\in \ms U\}
$$
be the ultrafilter limit prime ideal of $Y$, with respect to $\ms U$
(see \cite[Lemma 2.4]{Calota}).
Thus, for any element $x\in Y_{\ms
U}$, the set $V(x)\cap Y \in \ms U$, thus it is infinite since the ultrafilter is not trivial and, by assumption, $x=0$. This proves that $Y_{\ms U}=(0)$. Thus $(0)$ is a prime ideal and so $A$ is an integral domain. Furthermore, $(0)\in \ad^{\rm c}(Y)$. Since the equality $Y_{\ms U}=(0)$ holds for any non principal ultrafilter  $\ms U$ on $Y$, the conclusion follows immediately from the equality $(\diamond)$ at page \pageref{diamond}.
\end{proof}

\begin{oss}
Now we observe that the nontrivial part (i)$\Longrightarrow$(ii) of Griffin's characterization of Krull-type domains (Theorem \ref{griffin}(2)) follows from Theorem \ref{pvmdchar}. Suppose $D$ is a Krull-type domain and let $\mathcal V:=\{D_{\f p}:\f p\in Y\}$  be an essential and locally finite representation of $D$ (for some subset  $Y$ of $ \spec(D)$). Of course, for any $d\in D-\{0\}$, only finitely many prime ideals in $Y$ contain $d$. Thus, if $Y$ is infinite, by Lemma \ref{easy} we have 
$$\ad^{\rm c}(Y)=Y\cup \{0\}\subseteq \mathcal E(D):=\{\f p\in \spec(D):D_{\f p} \mbox{ is a valuation domain}\}
$$
If $Y$ is finite, it is clearly closed, since   the constructible topology is Hausdorff, in particular. Thus, in any case, we have $\ad^{\rm c}(Y)\subseteq \mathcal E(D)$ and, by Theorem \ref{pvmdchar}, $D$ is a P$v$MD.
\end{oss}  

\begin{cor}\label{sopranelli}
{ Let $K\subseteq L$ be an algebraic field extension, $A$ be a P$v$MD whose field of fractions is $K$ and $B$ be an integrally closed essential domain with field of fractions $L$. Moreover, suppose  that $B$
admits an essential representation $\mathcal V$ such that, for any
$V\in \mathcal V$, the center of $V$ in $A$ is a t-ideal. Then $B$
is a P$v$MD.}
\end{cor}

\begin{proof}
Let $X$ be the subset of $\spec(B)$ such that $\mathcal V=\{B_{\f h}:\f h\in X\}$, and let
$\iota^\star:\spec(B)\longrightarrow\spec(A)$ denote the map naturally induced by
the inclusion $\iota:A\longrightarrow B$. By
\cite[Chapter 3, Exercise 29]{AM}, the map $\iota^\star$ is continuous (and closed), if $\spec(A),\spec(B)$ are equipped with the constructible topology, and then $\iota^\star(\ad^{\rm c}(X))\subseteq \ad^{\rm c}(\iota^\star(X))$. On the other hand, $\iota^\star(X)$ is clearly the set of all centers in $A$ of the valuation domains in $\mathcal V$ and thus, keeping in mind assumption and applying (\cite[Lemma 2.4 and
Proposition 2.5]{Calota}, we have  $\ad^{\rm c}(\iota^\star(X))\subseteq\tspec(A)$. Now take a prime ideal $\f p\in \ad^{\rm c}(X)$. { Since we have $\iota^\star(\ad^{\rm c}(X))\subseteq \tspec(A)$, $\f p\cap A$ is a $t$-prime ideal of $A$ and, since $A$ is a P$v$MD, $A_{\f p\cap A}$ is a valuation domain such that $A_{\f p\cap A}\subseteq B_{\f p}$. Then the integral closure $\overline{A_{\f p\cap A}}^{(L)}$ of $A_{\f p\cap A}$ in $L$ is a Pr\"ufer domain whose field of fractions is $L$ and, since $B$ is integrally closed, we have $\overline{A_{\f p\cap A}}^{(L)}\subseteq B_{\f p}$. It follows that $B_{\f p}$ is a valuation domain, being it a local overring of a Pr\"ufer domain. Now it sufficies to apply Theorem \ref{pvmdchar}. }
\end{proof}

\begin{ex} Let $\SQ \subset K$ be a finite field extension and consider a DVR overring $(V,M_V)$ of $\SZ[X]$ such that $A = V \cap \SQ[X]$ is a P$v$MD (\cite[Theorem 5.8]{lt}). Let $(W,M_W)$ be an extension of $V$ to $K(X)$. Then $B = W \cap K[X]$ is a P$v$MD.

\medskip

We have to show that $W \cap \bigcap_{Q \in \Max(K[X])}K[X]_Q$ is an essential representaton of $B$. It is easy to check that $K[X]_Q = B_{Q \cap B}$, for each $Q \in \Max(K[X])$ (these ideals $Q \cap B$ are exactly the uppers to zero of $B$)

As regards $W$,  $M_W \cap B \supset M_W \cap A = M_V \cap A$ (since $W$ is an extension of $V$) and it is known that $A_{M_V \cap A} = V$ (\cite[Theorem 5.8 \& Lemma 1.3 (2)]{lt}). Then $V \subset B_{M_W \cap B}$ and so $B_{M_W \cap B}$ is a valuation domain since it contains the integral closure of $V$ in $K$, that is Pr\"ufer (\cite[Theorem 22.3]{gilmer}). For dimension consideration, it follows that 
$B_{M_W \cap B} = W$. 

Now $W$ is centered in $M_V \cap A$ that is a $t$-ideal of $A$ (since it is minimal over a principal ideal   by \cite[proof of Lemma 1.3 (1)]{lt}). All the valuation overrings of $K[X]$ are centered in the upper to zero primes of $A$, which are also $t$-primes of $A$.
Thus $B$ is a P$v$MD by Corollary \ref{sopranelli}.

\end{ex}

\begin{cor}{\rm (\cite[Corollary 3.9]{ka} and \cite[Proposition 5.1]{mo-za})}\label{sottointersezioni}
Let $A$ be a PvMD and let $X$ be a non empty collection of $t$-prime
ideals of $A$.
Then $\bigcap\{A_{\f p}:\f p\in X\}$ is a PvMD.
\end{cor}

\begin{proof}
Set $B:=\bigcap\{A_{\f p}:\f p\in X\}$ and, for any prime ideal $\f
p\in X$, set $\widetilde{\f p}:=\f pA_{\f p}\cap B$. Note that, since
obviously $B_{\widetilde{\f p}}=A_{\f p}$, for any $\f p \in X$, the
collection of rings $\mathcal V:=\{B_{\widetilde{\f p}}:\f
p\in X\}$ is an essential representation of $B$ such that $\widetilde{\f p}\cap A=\f p$ is a $t$-prime ideal of $A$.  Then the statement follows immediately by Corollary \ref{sopranelli},  just by taking $L:=K$.
\end{proof}

Now we give a sufficient condition for an intersection of a family of P$v$MDs to be a P$v$MD. Recall that a family $\mathcal F$ of subsets of subsets of a topological space $X$ is called \emph{a locally finite collection of sets} if for any $x\in X$ there is a neighborhood of $U$ of $X$ such that $\{F\in \mathcal F:F\cap U\neq \emptyset\}$ is finite. 

Let $\{D_i:i\in I\}$ a family of P$v$MDs and set $D:=\bigcap\{D_i:i\in I\}$. We say that $D$ \emph{essential, with respect to the family $\{D_i:i\in I\}$}, if the canonical representation 
$$
\{(D_i)_{\f q}:\f q\in \tspec(D_i), i\in I\}
$$
of $D$ is essential. It follows immediately that if $D$ is essential with respect to $\{D_i:i\in I\}$, then $D_i$ is an overring of $D$, for any $i\in I$.  
\begin{thm}\label{intersezioni}
Let $\{D_i:i\in I\}$ a nonempty collection of P$v$MDs set $D:=\bigcap\{D_i:i\in I\}$, and suppose that $D$ is essential, with respect to the family $\{D_i:i\in I\}$.  Assume also that for any $\f p \in \spec(D)$ there are an element $f\in D-\f p$ and a finitely generated ideal $\f a\subseteq \f p$ such that only for finitely many indices $i\in I$ may exist a t-prime ideal $\f q$ of $D_i$ such that $f\notin \f q$ and $\f a\subseteq \f q\cap D$. Then $D$ is a P$v$MD. 
\end{thm}
\begin{proof}
For any $i\in I$, let $\iota_i:D\longrightarrow D_i$ denote the inclusion. Of course, the set of the centers of the canonical and essential representation  
$$
\{(D_i)_{\f q}:\f q\in \tspec(D_i), i\in I\}
$$
of $D$ is $X:=\{\f q\cap D:\f q\in \tspec(D_i),i\in I\}=\bigcup_{i\in I}\iota_i^\star(\tspec(D_i))$. Let $\f p\in \spec(A)$. By assumption, the open neighborhood $D(f)\cap V(\f a)$ (with respect to the constructible topology) intersects $\iota_i^\star(\tspec(D_i))$ only for finitely many $i\in I$. Moreover, for any $i\in I$ the set $\iota_i^\star(\tspec(D_i))$ is closed, with respect to the constructible topology, being $\tspec(D_i)$ closed and $\iota_i^\star$ continuous. Thus $\{\iota_i^\star(\tspec(D_i)):i\in I\}$ is a locally finite family of closed sets of $\spec(D)$. By \cite[Theorem 1.1.11]{Engelking}, we infer that $X$ is closed, with respect to the constructible topology. Thus the conclusion follows immediately from Corollary \ref{pvmdcharcor}.
\end{proof}
The following results are immediate consequences of Theorem \ref{intersezioni}.

\begin{cor}\label{finite intersection}
Let $D_1,\z,D_n$ be P$v$MDs, set $D:=D_1\cap\z\cap D_n$ and assume that $D$ is essential, with respect to $\{D_1,\z,D_n\}$. Then $D$ is a P$v$MD. 
\end{cor}
\begin{cor}
Let $\{D_i:i\in I\}$ be a nonempty family of P$v$MDs, set $D:=\bigcap\{D_i:i\in I\}$ and suppose that $D$ is essential, with respect to $\{D_i:i\in I\}$. Assume that at least one of the following properties is satisfied.
\begin{enumerate}[\hspace{0.6cm}\rm (1)]
\item For any $\f p\in \spec(D)$ there is an element $f\in D-\f p$ such that for only   finitely many indices $i\in I$ there exists a t-prime ideal $\f q $ of $D_i$ such that $f\notin \f q\cap D$.
\item For any $\f p\in \spec(D)$ there is a finitely generated ideal $\f a$ of $A$ contained in $\f p$ such that for only   finitely many indices $i\in I$ there exists a t-prime ideal $\f q $ of $D_i$ such that $\f q\cap D\supseteq \f a$.
\end{enumerate}
Then $D$ is a P$v$MD. 
\end{cor}

The following example gives a direct application of Corollary \ref{finite intersection}.

\begin{ex}

Let $(V,M_V)$ be a one-dimensional, discrete valuation overring  of $\SZ[X]$ such that $V \cap \SQ[X]$ is P$v$MD, not Pr\"ufer (see \cite[Proposition 4.1 and Theorem 5.8]{lt}). Suppose  that $M_V \cap \SZ[X] = (p,f(X))$, where $p \in \SZ$ is a prime integer and $f(X) \in \SQ[X]$ is a non linear, monic and irreducible polynomial over $\mathbb F_p$ (the field with $p$ elements).

Then $A := V \cap \Int(\SZ)$ is a P$v$MD, not Pr\"ufer.

That $A$ is not Pr\"ufer follows from the fact that its overring $V \cap \SQ[X]$ is not Pr\"ufer.

We recall that all the prime ideals of $\Int(\SZ)$ are either $(0)$, uppers to zero or maximals of the type 
$\f m_{p,\alpha} = \{f \in \Int(\SZ) : f(\alpha) \in \widehat{p\SZ_{(p)}}\}$, where $p \in \SZ$ is prime and $\alpha \in \widehat{\SZ_{(p)}}$ (the $p$-adic completion of $\SZ$). It is also well-known that $\f m_{p,\alpha} \cap \SZ[X] = (p, X - a)$, where $a \in \SZ$ is such that $\alpha -a \in \widehat{p\SZ_{(p)}}$ (\cite[Remark V.2.6 (iiib)]{cc}). This implies that $\f m_{p,\alpha} \cap A \nsubseteq M_V \cap A$ since  $\f m_{p,\alpha} \cap \SZ[X] \nsubseteq M_V \cap \SZ[X]$ and $\SZ[X] \subset A$.

The domain $\Int(\SZ)$ is Pr\"ufer, so all its localizations at prime ideals are valuation domains.

Let's see   $V \cap (\bigcap_{Q \in \spec(\Int(\SZ))}\Int(\SZ)_Q)$ is an essential representation of $A$.

If $Q \in \spec(\Int(\SZ))$ and $Q \cap \SZ = (0)$, then $\Int(\SZ)_Q = A_{Q \cap A} = \SQ[X]_{(f)}$, where $f$ is such that $Q = f\SQ[X] \cap \Int(\SZ)$. Thus $\Int(\SZ)_Q$ is a localization of $A$ that is a valuation domain.

If $Q \cap \SZ = (p)$, for some prime $p \in \SZ$, then $Q = \f m_{p,\alpha}$, $\exists \,\, \alpha \in \widehat{\SZ_{(p)}}$. In this case, $A_{(\f m_{p,\alpha} \cap A)} = V_{(A \backslash (\f m_{p,\alpha} \cap A))} \cap \Int(\SZ)_{(A \backslash (\f m_{p,\alpha} \cap A))}$. But, since we have observed that $\f m_{p,\alpha} \cap A \nsubseteq M_V$, it follows that $V_{(A \backslash (\f m_{p,\alpha} \cap A))} = \SQ(X)$ and so  $A_{(\f m_{p,\alpha} \cap A)} = \Int(\SZ)_{(A \backslash (\f m_{p,\alpha} \cap A))}$, that is a valuation domain since $\Int(\SZ)$ is Pr\"ufer and $A_{(\f m_{p,\alpha} \cap A)}$ is a local overring of $\Int(\SZ)$.

Now, we'll see that $V$ is a localization of $A$ at some prime ideal. Obviously $V \nsupseteq \Int(\SZ)$, otherwise $V \cap \SQ[X]$ would be   Pr\"ufer  as being an overring of $\Int(\SZ)$. We also have that $V$ is rational (i.e. its value group is contained in $\SQ$).
By \cite[Lemma 1.3]{ho72}, we easily have that   $A_{\f M_V} = V$, where $\f M_V$ is the center of $V$ in $A$.

%

By Corollary \ref{pvmdcharcor} we have to show that the set of the centers in $A$ of $\{\Int(\SZ)_Q; Q \in \spec(\Int(\SZ))\} \cup \{V\}$  is closed  with respect to the constructible topology and this is equivalent to ask that the set of the centers in $A$  of $\{\Int(\SZ)_Q; Q \in \spec(\Int(\SZ))\} $ is closed. Now, this set is exactly the image of $\Int(\SZ)$ under the map 
$$f^{\star}: \spec(\Int(\SZ)) \rightarrow \spec(A), \quad P \mapsto P \cap A$$ 

and so it is closed. 
\end{ex}
\section{An application to Integer-Valued Polynomials}\label{applications}

Given a domain $D$ with quotient field $K$, the Integer-valued polynomial ring on $D$ is the ring $\Int(D) := \{f \in K[X]: f(D) \subseteq D\}$.

In \cite{Ta} (for Krull-type domains) and \cite{Calota} (for general domains), the authors study conditions on $D$ to have that $\Int(D)$ is a P$v$MD.

Following the notation of \cite{Calota}, a { $t$-prime} ideal $P \in \spec(D)$ is called \textit{int-prime} if $\Int(D)_{(D \backslash P)} \neq D_P[X]$ (in the following, for simplicity of notation, we will put $\Int(D)_P := \Int(D)_{(D \backslash P)}$, for any prime ideal $P$ of $D$). 

For any domain $D$, it is well-known that 
$D = \bigcap_{P \in \tspec(D)}D_P.$ 

We define the following two subsets of $\tspec(D)$:
$$\Lambda_1 := \{P \in \tspec(D): \Int(D)_P = D_P[X]\}$$
and
$$\Lambda_0 := \{P \in \tspec(D): \Int(D)_P \neq D_P[X]\}.$$ 

From \cite[Proposition I.3.4]{cc} it follows that the ideals of $\Lambda_0$ are also maximal (since, by \cite[Corollary 1.3]{Calota},  $\mid D/P \mid < \infty$).

\medskip
We set
$D_1 := \bigcap_{P \in \Lambda_1}D_P$ and  $D_0 := \bigcap_{P \in
\Lambda_0}D_P$.
From \cite[Lemma 4.1]{Calota} it follows that
$$\Int(D) = D_1[X] \cap \Int(D_0).$$ 

\medskip

If $\Int(D)$ is a P$v$MD, then $\Int(D_0)$ is Pr\"ufer (\cite[Corollary 4.9]{Calota}), but this last condition is not sufficient to get that $\Int(D)$ is a P$v$MD, also assuming that $D$ is a P$v$MD (\cite[Example 5.1]{Calota}).
 
If $D$ is Krull-type, the conditon  $\Int(D_0)$ is Pr\"ufer is equivalent to ask that $\Int(D)$ is a P$v$MD. This result is implicitely shown in \cite{Ta}, but we give a more explicit proof of this fact in the next Theorem.
 
 \begin{thm}\label{krull-type}
Let $D$ be a Krull-type domain. Then $\Int(D)$ is a P$v$MD if and only if $\Int(D_0)$ is Pr\"ufer.
\end{thm}

\begin{proof}
If $\Int(D)$ is a P$v$MD, then we have already observed above that $\Int(D_0)$ is Pr\"ufer.

Suppose that $\Int(D_0)$ is Pr\"ufer. Then $D_0$ is almost Dedekind by \cite[Proposition VI.1.5]{cc}. If $P \in \Lambda_0$, then (by construction) $D_P$ is a local overring of $D_0$, and thus it is a DVR (as being $D_0$ almost Dedekind). So $P$ is height-one. From \cite[Theorem 3.2]{Ta} we know that when $D$ is Krull-type,  $\Int(D)$ is a P$v$MD if and only if each   $P \in \Lambda_0$ has height one. It follows that $\Int(D)$ is a P$v$MD.
\end{proof}

Using Theorem \ref{pvmdchar}, we will show that in Theorem \ref{krull-type} the Krull-type condition can be replaced by the weaker condition $\Int(D)_P = \Int(D_P)$, for each $P \in \tspec(D)$ (this is always verified when $D$ is Krull-type by \cite[Proposition 2.3]{Ta}).

\medskip

We recall {  several} facts  that we will freely use in the following:
 
\begin{oss}\label{t-ideals}
Let $D$ be an integral domain. 
\begin{enumerate} 
\item If  $S$ is a multiplicative subset of $D$, then each contraction to $D$ of a $t$-ideal of $D_S$ is a $t$-ideal of $D$ (\cite[Lemma 3.17]{ka}).
\item Let $Y$ be a nonempty collection of prime ideals of $D$ and let $D':=\bigcap\{D_{\f p}:\f p\in Y\}$. By applying \cite[Proposition 1.3]{ga} it follows that if $\f a$ is a $t$-ideal of $D'$, then $\f a\cap D$ is a $t$-ideal of $D$. 
\item A prime ideal of $D$ which is minimal over a principal ideal is a $t$-ideal (\cite[Corollaire 3, p. 31]{jaffard}).
In particular, in polynomial rings, the uppers to zero primes are always $t$-ideals. 

\end{enumerate}
\end{oss}

 \begin{lem}\label{dvr overring}
Let $V \subseteq W$ be valuation domains { having the same quotient field}, and suppose that $W$ has  finite residue field. Then $V=W$.
 \end{lem}
 
 \begin{proof}
 
 There exists a prime ideal $P$ of $V$ such that $V_P = W$. If $M_W$ is the maximal ideal of $W$, then $M_W \cap V = P$ and so $V/P \subseteq W/M_W$. But $W/M_W$ is finite, whence $V/P$ is a field. Then  $P$ is maximal in $V$ and $V = V_P = W$. \end{proof}

 \begin{prop}\label{almost Dedekind} With the above notation, let $D$ be a P$v$MD and   $D_0$ be a Pr\"ufer domain with finite residue fields. Suppose that $\Int(D)_{P} = \Int(D_{P})$, for
each $t$-maximal ideal $P$ of $D$. Let $i : D \hookrightarrow D_0$ be the inclusion map and $i^\star: \spec(D_0) \rightarrow \spec(D)$ be the induced contraction  sending $\f q \mapsto \f q \cap D$. Then $i^\star(\spec(D_0)) = \Lambda_0$. In particular, it follows that $\Lambda_0$ is closed with respect to the constructible topology.
 
\end{prop}

\begin{proof}
Let $\f q \in \spec(D_0)$ and set $P := \f q \cap D$. 
Since $D_0$ is a Pr\"ufer domain, $P$ is a $t$-prime ideal of $D$, by Remark \ref{t-ideals} (ii,iv). Keeping in mind that $D$ is a P$v$MD, $D_{P}$ is a valuation domain such that $D_{P}\subseteq (D_0)_{\f q}$. Furthermore, by Lemma \ref{dvr overring}, we have $D_{P}=(D_0)_{\f q}$, since, by assumption, the residue field of $(D_0)_{\f q}$ is finite. 
Thus we have $\Int(D)_{P} = \Int(D_{P})$ and $\Int(D_{P}) \neq D_{P}[X]$ (\cite[Proposition I.3.16]{cc}). So $P \in \Lambda_0$. This proves thaq $i^\star(\spec(D_0))\subseteq \Lambda_0$. The converse inclusion is trivial. The fact that $\Lambda_0$ is closed with respect to the constructible topology is now clear, in view of \cite[Chapter 3, Exercise 27]{AM}.
\end{proof}

\begin{oss}
{ The last statement of} Proposition \ref{almost Dedekind} stricty generalizes \cite[Lemma 2.6]{Calota}, in which the same result is shown for domains $D$ such that $\Int(D)$ is a P$v$MD.
\end{oss}
 
 \begin{lem}\label{iota0} With the above notation, suppose that $D$ is a P$v$MD such that $\Int(D)_P = \Int(D_P)$, for
each $t$-maximal ideal $P$ of $D$, and that  $D_0$ is almost Dedekind with all finite residue fields. 
Let $i_0: \Int(D) \rightarrow \Int(D_0)$ be the inclusion map and $i_0^\star: \spec(\Int(D_0)) \rightarrow \spec(\Int(D))$ be the induced contraction map sending $Q \mapsto Q \cap \Int(D)$.
Then $$\{M \cap \Int(D) :    M \in \spec(\Int(D_P)), \, P \in \Lambda_0\} = i_0^\star(\spec(\Int(D_0)).$$
 
 \end{lem}
 
 \begin{proof} 
 
 We observe  that if $P \in \Lambda_0$, then $D_P = (D_0)_{\f q}$, for some $\f q \in \spec(D_0)$  
 (Proposition \ref{almost Dedekind}). In particular, $D_P = (D_0)_{D \backslash P}$. Thus $D_P$ is also a localization of $D_0$.
 Then $\Int(D_P)   \supseteq \Int(D_0)$, whence we have the inclusion 
 $\{M \cap \Int(D) :    M \in \spec(\Int(D_P)), \, P \in \Lambda_0\} \subseteq \iota_0^\star(\spec(\Int(D_0)). $
 
 Conversely, let $Q \in \spec(\Int(D_0))$. Then   $Q \cap D \in \Lambda_0$. In fact $P  = Q \cap D = (Q \cap D_0) \cap D$. By Proposition \ref{almost Dedekind}, $P \in \Lambda_0$. So, $Q = Q\Int(D_0)_{D \backslash P} \cap \Int(D_0)$. It is easy to check that $\Int(D_0)_{D \backslash P} = \Int(D)_P$. Since $\Int(D_P) = \Int(D)_P$, the thesis follows.
 \end{proof}

\begin{thm}\label{pvmdint}
With the above notation, let $D$ be an integral domain such that $\Int(D)_P = \Int(D_P)$, for
each $t$-maximal ideal $P$ of $D$. Then the following conditions are equivalent:
\begin{enumerate}
\item $\Int(D)$ is a P$v$MD;

\item if  $D$ is a P$v$MD and $\Int(D_0)$ 
is a Pr\"ufer domain.
\end{enumerate}  
 
\end{thm}

\begin{proof}

  (1)$\Rightarrow$ (2) It is already known.

(2)$\Rightarrow$ (1)
Since $D$ is a P$v$MD, also $D_1$ and $D_1[X]$ are P$v$MDs (\cite[Corollary 3.9 and Theorem 3.7]{ka}). Moreover $\Int(D_0)$ is Pr\"ufer, whence it is a P$v$MD. So $\Int(D)  = D_1[X] \cap \Int(D_0)$ is the intersection  of two P$v$MDs and, by Corollary \ref{finite intersection}, it is suffcient to show that $\Int(D)$ is essential with respect to both $D_1[X]$ and $\Int(D_0)$.

Take $Q \in \tspec(D_1[X])$, $Q \cap D \neq (0)$. By Remark \ref{t-ideals}(2), $Q \cap \Int(D) \in \tspec(\Int(D))$ and by \cite[Proposition 2.1]{Ta} $\f p := Q \cap D \in \tspec(D)= \Lambda_0 \cup \Lambda_1$.

If $\f p \in \Lambda_0$, then $\Int(D)  \nsubseteq D_{\f p}[X]$. But $\Int(D) \subseteq D_1[X] \subseteq D_{\f p}[X]$, which is a contraddiction. It follows that $\f p \in \Lambda_1$.  We observe that, for such a $\f p$, $(D_1)_{D \backslash \f p} = D_{\f p}$ since $D_1 \subseteq D_{\f p}$.

Now 
$$D_1[X]_Q = (D_1[X]_{D \backslash \f p})_{Q^e} = D_p[X]_{Q^e} = (\Int(D)_{\f p})_{Q^e}.$$

If  $Q$ is an upper to zero ideal (in this case   $QK[X] =  fK[X]$ for some irreducible $f \in K[X]$), then $D_1[X]_Q =  K[X]_f = \Int(D)_{Q \cap \Int(D)}$.

Then $D_1[X]$ is essential with respect to $\Int(D)$.

With regards to $\Int(D_0)$, take $M \in \spec(\Int(D_0))$. By Lemma \ref{iota0} $M \cap \Int(D) = M' \cap \Int(D)$, where $M' \in \spec(\Int(D_{\f p}))$, with $\f p \in \Lambda_0$. Then $\Int(D_0)_M = (\Int(D_0)_{\f p})_{\Int(D_0) \backslash M} = (\Int(D)_{\f p})_{\Int(D_0) \backslash M} = \Int(D)_{M \cap \Int(D)}$. Thus $\Int(D)$ is essential also with respect to $\Int(D_0)$ and the thesis follows.

\end{proof}

\end{document}